\documentclass[11pt]{amsart}

\makeatletter
\def\emailaddrname\@ifnotempty#1#2\/:\space{}
\makeatother

\newtheorem{theorem}{Theorem}
\newtheorem{lemma}{Lemma}

\title{A note on the Slater--V{\'e}lez permutation}

\author{Andreas Floros}
\address{Imperial College London}
\email{andreas.floros18@imperial.ac.uk}

\begin{document}

\begin{abstract}
    We study the greedy sequence of positive integers introduced by Slater and V{\'e}lez. They proved that it is a permutation of the positive integers but left open whether its sequence of consecutive absolute differences is also a permutation. We prove that it is.
\end{abstract}

\maketitle

Let $\mathbb{N}$ denote the positive integers. Starting with $a_1=1$ and $a_2=2$, continue in a greedy way, choosing $a_{n+1}$ to be the least unused positive integer for which $|a_{n+1}-a_n|$ differs from all earlier consecutive absolute differences.

Slater and V{\'e}lez \cite{SlaterVelez1977} proved $(a_n)_{n\ge1}$ is a permutation of $\mathbb{N}$ and asked whether its consecutive absolute differences also exhaust $\mathbb{N}$. We show that this is true.
\begin{theorem}\label{thm:slater-velez}
The sequence $(|a_{n+1}-a_n|)_{n\ge1}$ is a permutation of $\mathbb{N}$.
\end{theorem}
For any permutation of $\mathbb{N}$, let $s(v)$ be the term immediately following $v$. Call $v$ an \emph{ascent vertex} if $s(v)>v$ and a \emph{descent vertex} if $s(v)<v$. Define
\[
 d(v)=|s(v)-v|,\qquad A(X)=\#\{v\le X:s(v)>v\}.
\]
We highlight some basic facts. Here, the successor map, $s$, is injective since each vertex has at most one predecessor. If consecutive absolute differences are distinct, $d$ is also injective. Moreover, there are infinitely many ascent vertices, otherwise the permutation would eventually be strictly decreasing.

In the greedy permutation of interest, we show that a missing difference pairs every sufficiently large ascent vertex with a descent a fixed distance above it. This forces $A(X)\le X/2+C$, for a fixed $C$, contrary to the following.
\begin{lemma}\label{lem:ascent-count}
For any permutation of $\mathbb{N}$ with distinct consecutive absolute differences,
$A(X)-X/2$ is unbounded above.
\end{lemma}
\begin{proof}
Suppose instead that $A(X)\le X/2+C$ for every positive integer $X$, where $C$ is a fixed positive integer. To remove the additive constant, set aside the $C$ smallest ascent vertices. Call this fixed set $F$ and list the other ascent vertices as $u_1<u_2<\dots$. Then, invoking the assumed bound yields
\begin{equation}\label{eq:vertex-bound}
 C+i=A(u_i)\le u_i/2+C\implies u_i\ge2i.
\end{equation}
Fix an integer $X\ge\max F$. Let $P=\{u_i:u_i\le X\}$ and
$Q=\{1,\ldots,X\}\setminus P$. Write $p=|P|$ and $q=|Q|$.
We now estimate the sum in each set. By \eqref{eq:vertex-bound}
\begin{equation}\label{eq:ascent-sum}
 \sum_{v\in P}v\ge\sum_{i=1}^p2i=p(p+1).
\end{equation}
$Q$ consists of descent vertices and the fixed set $F$. For a descent vertex, $v$, we have $s(v)+d(v)=v$. Instead, for $v\in F$ we have $s(v)+d(v)=v+2d(v)$. The extra contribution from $F$ is therefore $E=2\sum_{v\in F}d(v)$, which is independent of $X$. Now, the setup implies that both $s$ and $d$ are injective. Consequently, the images $s(Q)$ and $d(Q)$ each contain $q$ distinct positive integers. Hence we have the lower bound
\begin{equation}\label{eq:complement-sum}
 \sum_{v\in Q}v+E
 =\sum_{v\in Q}s(v)+\sum_{v\in Q}d(v)
 \ge2\sum_{i=1}^qi=q(q+1).
\end{equation}
Adding \eqref{eq:ascent-sum}, \eqref{eq:complement-sum}, using
$p+q=X$ and $p^2+q^2\ge X^2/2$ gives
\[
 \frac{X(X+1)}2+E
 \ge p(p+1)+q(q+1)
 \ge\frac{X(X+2)}2.
\]
Thus $X\le2E$ for arbitrarily large $X$, a contradiction.
\end{proof}

\begin{proof}[Proof of Theorem~\ref{thm:slater-velez}]
Toward a contradiction, assume that a positive integer, $h$, never occurs as a consecutive absolute difference of $(a_n)_{n\ge1}$.
Since $d$ is injective by construction, only finitely many vertices satisfy $d(v)\le h$. Fix an integer $L>1$ beyond them, so that we have $d(v)>h$ whenever $v\ge L$.

Let $v\ge L$ be an ascent vertex. Then $s(v)>v+h$. When the sequence reaches
$v$, the value $v+h$ already appeared. Otherwise it would be an unused value smaller than the chosen successor and with unused difference $h$.

At that earlier occurrence of $v+h$, the value $v$ was still unused and was
a legal choice. The greedy rule then implies
\[
 s(v+h)\le v < v + h.
\]
That is, $v+h$ is a descent vertex.

Based on the above observations, for $X\ge L$, the ascent vertices in $[L,X]$ and their translates by $h$
form two disjoint sets in $[L,X+h]$. Each set has $A(X)-A(L-1)$ members. Therefore
\[
 2\bigl(A(X)-A(L-1)\bigr)\le X-L+h+1.
\]
This shows $A(X)-X/2$ is bounded above, contradicting
Lemma~\ref{lem:ascent-count}.
\end{proof}

\section*{Acknowledgments}
The author thanks Neil J. A. Sloane for taking an interest in this note.

\bibliographystyle{amsplain}
\bibliography{\jobname}

\providecommand{\bysame}{\leavevmode\hbox to3em{\hrulefill}\thinspace}
\providecommand{\MR}{\relax\ifhmode\unskip\space\fi MR }
\providecommand{\MRhref}[2]{%
  \href{http://www.ams.org/mathscinet-getitem?mr=#1}{#2}
}
\providecommand{\href}[2]{#2}
\begin{thebibliography}{1}

\bibitem{SlaterVelez1977}
P.~J. Slater and W.~Y. V{\'e}lez, \emph{Permutations of the positive integers with restrictions on the sequence of differences}, Pacific Journal of Mathematics \textbf{71} (1977), no.~1, 193--196.

\end{thebibliography}

\end{document}